\documentclass[11pt,reqno]{amsart}

\usepackage{cite} 
\usepackage{color}
\usepackage{graphicx}         
\usepackage{amsmath,amssymb}
\usepackage{subfigure}
\usepackage{amssymb}
\usepackage{amsfonts}
\usepackage{eucal}
\usepackage[titletoc,title]{appendix}
\usepackage{mathptmx} 
\usepackage{epsfig}
\usepackage{epstopdf}
\usepackage{subcaption} 
\usepackage{hyperref}
\usepackage{microtype}
\usepackage{fullpage}

\usepackage{tikz} \usetikzlibrary{calc} \tikzset{>=latex}

\DisableLigatures{encoding = *, family = * }

\usepackage{bbm}
\usepackage{mathrsfs}

\newtheorem{theorem}{Theorem}
\newtheorem{definition}{Definition}

\newtheorem{lemma}{Lemma}
\newtheorem{remark}{Remark}

\numberwithin{equation}{section}

\usepackage{color} 
\allowdisplaybreaks
\usepackage{flushend}
\usepackage{verbatim} 

\title{Operator learning for prescribed-time stabilization of reaction-diffusion systems}
\subjclass[2020]{}

\author{Kaijing Lyu\textsuperscript{\,$\dagger$}}
\address{\textsuperscript{$\dagger$}\, School of Mathematics and Statistics,  Beijing Institute of Technology, 100081 Beijing, China.
	\newline \indent \textsuperscript{$\ast$}\, Chair of Computational Mathematics, DeustoTech, University of Deusto, Avenida de las Universidades 24, 48007 Bilbao, Basque Country, Spain
} 
\email{kjlv@bit.edu.cn}

\thanks{This project has received funding from the European Research Council (ERC) under the European Union's Horizon 2030 research and innovation programme (grant agreement NO: 101096251-CoDeFeL). UB was partially supported by the Grant PID2023-146872OB-I00-DyCMaMod of MICIU (Spain) and by the COST Actions CA24122 - multiscale Stochastics, Patterns, and Analysis of Combinatorial Environments and CA24136 - Interactions between Control Theory and Machine Learning.}

\author{Umberto Biccari\textsuperscript{\,$\ast$}} 
\email{umberto.biccari@deusto.es}

\author{Jun-Min Wang\textsuperscript{\,$\dagger$}}  
\email{jmwang@bit.edu.cn}

\begin{document}      

\begin{abstract}
This paper addresses boundary prescribed-time stabilization of a one-dimensional heat equation with spatially and temporally varying coefficients. In contrast to asymptotic or exponential stabilization, prescribed-time stabilization ensures convergence to equilibrium within a user-defined time that is independent of the initial condition, a property that is particularly attractive in applications with stringent transient performance requirements. The backstepping design for this problem requires solving, at each time instant, a two-dimensional time-dependent kernel Partial Differential Equation (PDE) whose solution continuously varies with the plant coefficients. The repeated numerical solution of this parabolic kernel PDE results in a prohibitive computational burden, thereby limiting real-time applicability. To overcome this limitation, we propose a neural-operator-based approximation of the mapping from the time-varying system coefficient to the corresponding backstepping kernel. The operator is trained offline using representative solutions of the kernel PDE and subsequently deployed online to generate the required time-varying kernels in real time. We establish, via Lyapunov analysis, that the resulting neural-operator-based controller preserves prescribed-time stability provided that the operator approximation error satisfies an explicit bound. Furthermore, we investigate a direct approximation of the full feedback law mapping the plant parameter functions and state measurements to the boundary control input. For this setting, we prove semiglobal practical prescribed-time stability of the closed-loop system. Numerical experiments demonstrate that the proposed approach reduces the computational cost of kernel generation by several orders of magnitude, thereby enabling real-time prescribed-time stabilization for heat equations with spatially and temporally varying coefficients.
\end{abstract}

\maketitle 

\section{Introduction}

Boundary stabilization of parabolic partial differential equations has been extensively developed through constructive feedback design methodologies aimed at ensuring asymptotic or exponential decay of solutions. Among these, backstepping transformations have emerged as a systematic and powerful framework for deriving explicit boundary feedback laws for reaction–diffusion systems \cite{smyshlyaev2005control,krstic2008backstepping,meurer2009tracking}. The method relies on a Volterra transformation whose kernel solves an auxiliary PDE posed on a triangular domain, yielding well-posed closed-loop systems with precise decay estimates. These techniques have been successfully extended to settings involving variable coefficients, nonstandard boundary conditions, delay effects, and output-feedback configurations \cite{espitia2019boundary,espitia2021predictor,zekraoui2025output,zhang2025prescribed}.

In this paper, we consider a one-dimensional reaction–diffusion equation on a bounded interval with a Robin boundary condition at one end and a Dirichlet boundary actuation at the other, where the reaction coefficient depends on both space and time. While for time-invariant coefficients the associated kernel equation is stationary and can be solved offline \cite{smyshlyaev2005control}, the time-varying case leads to a two-dimensional parabolic evolution equation for the kernel. Although well-posedness results and constructive designs are available under suitable regularity assumptions \cite{meurer2009tracking}, the kernel must then be updated continuously over the time horizon. At this regards we shall stress that, when the reaction coefficient depends on both space and time, the associated kernel equation generally does not admit a closed-form analytical solution and must be solved numerically. This yields the repeated numerical solution of a two-dimensional PDE, which significantly increases the computational burden and poses a major challenge for real-time implementation. 

The need for real-time implementability is not merely technical, but intrinsic to many distributed-parameter control applications. In large-scale networked and physical systems, control inputs must be computed at rates compatible with the system dynamics and available sensing and actuation infrastructure. For instance, traffic flow regulation based on hyperbolic PDE models requires control updates at operational time scales dictated by vehicle flow and measurement frequencies \cite{yu2019traffic}.

Similarly, in spatially distributed chemical or thermal processes described by reaction–diffusion equations, boundary control actions must be computed online using current state measurements. In such settings, repeatedly solving a two-dimensional kernel PDE during operation may be computationally prohibitive, particularly under fine spatial discretizations or embedded hardware constraints. Comparable considerations arise in continuum models of multi-agent systems \cite{meurer2011finite}, tubular reactors \cite{orlov2009discontinuous}, and propulsion or thermal management systems \cite{smyshlyaev2004closed}, where feedback laws must be evaluated reliably and efficiently within strict timing constraints.

These observations highlight that, beyond theoretical feasibility of the feedback law, its computational realizability in real time is a central practical concern. In this framework, the main contribution of the present work is to address this computational bottleneck. We employ a Neural Operator (NO) to learn, offline, the mapping from the time-varying reaction coefficient to the corresponding backstepping kernel. Once trained, the NO provides a fast surrogate for kernel evaluation, thereby eliminating the need for repeated online solution of the kernel PDE. We further show that, under suitable approximation bounds, the resulting closed-loop system preserves the desired terminal steering property. In this sense, the neural-operator approximation is not merely a numerical acceleration tool, but an integral component of a control design equipped with rigorous guarantees.

\subsection{Neural-operator approximation for prescribed-time backstepping}

Recent works have shown that NOs can successfully approximate gain kernels arising in backstepping-based PDE control. In particular, NOs have been employed to bypass repeated gain computations and accelerate controller implementation \cite{bhan2023operator,bhan2024adaptive2}. Among operator-learning architectures, DeepONet \cite{lu2021learning,Approximation} has been used to approximate backstepping kernels for $2\times2$ hyperbolic systems \cite{wang2025backstepping}. Related developments include neural-operator-based adaptive control, gain scheduling, and delay-compensated backstepping for various classes of PDEs \cite{bhan2023neural,krstic2024neural,lyu2025robust,wang2025deep,lamarque2025gain,lv2025neural01}.

These contributions demonstrate that operator learning can accurately reproduce stationary or parameter-dependent kernel maps, enabling efficient implementation of backstepping controllers. However, the existing literature primarily addresses time-invariant or slowly varying kernels that can be computed offline and subsequently approximated once.

The present work considers a substantially different setting: prescribed-time stabilization of a reaction–diffusion equation with spatially and temporally varying coefficient. Prescribed-time designs require a time-varying backstepping transformation, leading to a kernel that satisfies a two-dimensional parabolic evolution equation. In contrast to stationary kernel equations, this kernel evolves continuously over the time horizon and must, in principle, be recomputed at each time instant. This renders classical backstepping implementations computationally prohibitive for real-time applications.

Building upon the time-varying boundary feedback and predictor-based fixed-time designs developed in \cite{espitia2019boundary,espitia2021predictor}, we construct a prescribed-time stable target system and transfer this property to the original plant through a characterization of the bounded invertibility of the associated time-dependent backstepping transformation. The resulting design inherently involves genuinely time-evolving gain kernels.

To overcome the induced computational bottleneck, we approximate, via DeepONet, the operator that maps the time-dependent reaction coefficient to the corresponding time-varying backstepping kernel. The neural operator is trained offline using numerically generated kernel solutions and subsequently evaluated online to provide real-time kernel predictions, thereby eliminating repeated solution of the kernel PDE while preserving the backstepping structure.

Beyond computational acceleration, we provide a rigorous stability analysis of the resulting closed-loop system. Using Lyapunov methods combined with neural-operator approximation bounds, we show that the prescribed-time stability guarantees of the nominal backstepping design are preserved under sufficiently small operator approximation errors. We further extend the framework to direct approximation of the full feedback map, establishing semiglobal practical prescribed-time stability.

To the best of our knowledge, this is the first work that combines prescribed-time backstepping for parabolic PDEs with neural-operator approximation of a genuinely time-evolving kernel, together with explicit stability guarantees.

\subsection{Contributions and organization}

The main contributions of this paper are summarized as follows:

\begin{itemize}
    \item We develop a neural-operator-based framework for the prescribed-time stabilization of a heat equation with spatially and temporally varying reaction coefficient. The prescribed-time design induces a genuinely time-evolving backstepping kernel governed by a two-dimensional parabolic PDE. To avoid repeated online solution of this kernel equation, we approximate, via a NO, the mapping from the time-dependent coefficient to the corresponding backstepping kernel. 
    \item We establish rigorous prescribed-time stability guarantees for the closed-loop system obtained by replacing the exact kernel with its neural-operator approximation. Extending the Lyapunov analysis in \cite{espitia2019boundary} to encompass operator-approximation errors, we derive explicit bounds ensuring that the prescribed-time stability properties of the nominal backstepping design are preserved under sufficiently small approximation errors.
    \item Beyond kernel approximation, we investigate direct neural-operator approximation of the full prescribed-time boundary feedback map from plant parameters and state measurements to the control input. For this implementation-oriented design, we prove semiglobal practical prescribed-time stability and demonstrate significant reduction of the online computational burden.
\end{itemize}

The remainder of the paper is organized as follows. Section~\ref{sec:problem_statement} introduces the heat equation under consideration and presents the prescribed-time backstepping design.
Section~\ref{sec:NO} establishes the existence of a NO approximating the time-varying backstepping kernel. Section~\ref{sec:NO_stabilization} provides the Lyapunov-based stability analysis of the closed-loop system under the neural-operator-approximated kernel, proving preservation of prescribed-time stability. Section~\ref{sec:NO-approximated backstepping control law} investigates direct neural-operator approximation of the full feedback law and establishes semiglobal practical prescribed-time stability. Section~\ref{sec:simulations} presents numerical experiments validating the theoretical results and illustrating the computational efficiency of the proposed approach. Finally, Section~\ref{sec:conclusions} concludes the paper and outlines directions for future research.

\textbf{Notation:} In what follows, for all $p\in[1,+\infty]$ we will denote by $\|\cdot\|_p$ and $\|\cdot\|_{p,T}$ the standard norms over the spaces $L^p(0,1)$ and $L^p((0,1)\times (0,T))$, respectively. 

\section{Problem formulation}\label{sec:problem_statement}
In this paper, we consider the following reaction-diffusion systems with space- and time-dependent coefficient
\begin{equation}\label{eq:heat_1}
	\begin{cases} 
		v_t = \theta v_{xx} + \lambda v, & (x,t)\in (0,1)\times (0,T)
		\\
		v_x(0, t) = qv(0,t), & t\in(0,T)
		\\
		v(1, t) = U(t), & t\in (0,T)
		\\
		v(x,0) = v_0(x), & x\in (0,1) 
	\end{cases}
\end{equation}

Here, $T>0$ is a given time horizon for the PDE state $v(x,t): (0,1) \times (0,T) \rightarrow \mathbb{R}$, $U(t): (0,T)\to\mathbb{R}$ is a $C^\infty(0,T)$ control function, $\theta >0$ and $q>0$ are given constants, and 
\begin{align*}
    \lambda(x,t) \in G_\alpha(0,T;C^0(0,1)) \quad\text{ with } \alpha \in [1,2],
\end{align*}
is a variable coefficient belonging to the Gevrey class $G_\alpha$ defined as follows.
\begin{definition}(Gevrey class\cite{rodino1993linear})
Let $\alpha\geq 1$, $X$ be a Banach space, and fix a constant $D>0$. We say that a function $f:(0,T)\to X$ belongs to the Gevrey class of order $\alpha$, and we write $f\in  G_\alpha(0,T;X)$, if 
\begin{equation}\label{eq:Gevrey_norm}
    \begin{array}{l}
        f\in C^\infty(0,T;X)
        \\
        \displaystyle\|f\|_{G_\alpha}=\sup_{n\in\mathbb{N}^\ast} \left(\frac{\sup_{t\in (0,T)}\|\partial_t^n f(t)\|_X}{D^{n+1} (n!)^{\alpha}}\right)<+\infty. 
    \end{array}
\end{equation}
\end{definition}

Notice that, with $U\in C^\infty(0,T)$ and $\lambda \in G_\alpha(0,T;C^0(0,1))$, for any $v_0\in L^2(0,1)$ \eqref{eq:heat_1} admits a unique mild solution $v\in C(0,T;L^2(0,1))$ and, in particular, $v(\cdot,t)\in H^2(0,1)$ for any $t>0$.

Our main concern in this work will be the prescribed-time stabilization of \eqref{eq:heat_1}, that is, the design of a boundary feedback control $U(t)$ so that for prescribed time $T>0$ and initial datum $v_0$ the associated solution of \eqref{eq:heat_1} satisfies
\begin{align}\label{eq:stab}
    \lim_{t\to T} v(x,T)=0, &\quad \text{ for all } x\in(0,1)        
\end{align}

To this end, we will adopt the classical backstepping approach. This method is based on employing the following Volterra-type transformation
\begin{align}\label{eq:backstepping}
	w(x,t) = \mathcal{K}(t)[v(t, \cdot)](x) =v(x,t) - \int_0^x k(x, y, t)v(y,t)\,dy, 
\end{align}
where the kernel $k(x,y,t)$ is to be determined. This transformation is triangular in space and, under suitable boundedness conditions on $k$, defines an invertible mapping between $v$ and $w$.

The idea of backstepping is to choose the kernel $k$ and the boundary input $U(t)$ so that the transformed variable $w$ satisfies a target system with prescribed stability properties. In particular, we aim to obtain
\begin{equation}\label{eq:target1}
	\begin{cases}
		w_t = \theta w_{xx} - c(t)w, & (x,t)\in (0,1)\times (0,T)
		\\
		w_x(0,t) = qw(0,t), & t\in (0,T) 
		\\
		w(1,t) = 0, & t\in (0,T)
		\\
		  w(x,0) =w_0(x), & x\in (0,1)
	\end{cases}
\end{equation}
where 
\begin{align*}
    w_0(x) = v_0(x)-\int_0^x k(x,y,0)v_0(y)\,dy
\end{align*}
and $c(t)$ is a positive design function whose properties will be specified later.

Equation \eqref{eq:target1} has been obtained by differentiating \eqref{eq:backstepping} with respect to time and space and substituting the original equation \eqref{eq:heat_1}. By requiring that the resulting expression coincides with the right-hand side of \eqref{eq:target1} for all admissible states $v$, we obtain the kernel equation
\begin{align}\label{eq:kernel_1}
	\begin{cases}
		k_t =\theta\big(k_{xx}- k_{yy}\big)  - \gamma(y,t)k, & (x,y,t)\in\mathcal P
		\\
		k_y(x, 0, t) = qk(x,0,t)  
		\\
		\displaystyle k(x, x, t) = -\frac{1}{2\theta}\int_{0}^{x}\gamma(\xi,t)d\xi
	\end{cases}
\end{align}
with 
\begin{align*}
	\gamma(x,t) = \lambda(x,t)+c(t) 
\end{align*}
and where we have defined 
\begin{align*}
    \mathcal{P}  = \{(x, y, t) \in \mathbb{R}^2 \times (0, T): 0 \leq y \leq x \leq 1\}.    
\end{align*}

Thus, the kernel is constructed so that the reaction term $\lambda(x,t)v$ in the original system \eqref{eq:heat_1} is transformed into the damping term $-c(t)w$ in the target system \eqref{eq:target1}.
In addition, the control input $U(t)$ in \eqref{eq:heat_1} becomes 
\begin{equation}\label{eq:U}
	U(t) =  \int_{0}^{1}k(1,y,t)v(y,t)\,dy, 
\end{equation}
which depends only on the state $v$. Although the stability analysis is carried out in terms of the transformed variable $w$, the implemented control law is expressed in terms of $v$, since $w$ is introduced solely as a design and analysis tool.

The role of the coefficient $c(t)$ is central in the prescribed-time design. In the target system \eqref{eq:target1}, the term $-c(t)w$ acts as a distributed damping mechanism, and $c(t)$ has to be designed so to enforce increasingly strong damping near the prescribed time $T$. As we shall see, this is ensured by the following properties,
\begin{equation}\label{eq:c(t)}
    c\in C^\infty(0,T) \quad\text{ with }\quad
    \begin{array}{l}
        c(t)>0 
        \\
        \displaystyle\lim_{t \to T} \int_0^t c(\tau)d \tau =+\infty. 
    \end{array}
\end{equation}

Under these conditions \eqref{eq:c(t)}, Lyapunov analysis then shows that $w$ solution to \eqref{eq:target1} satisfies \eqref{eq:stab}. Finally, since the backstepping transformation \eqref{eq:backstepping} is invertible and bounded, the stabilization of $w$ implies the stabilization of the original state $v$.

It is worth noticing that, unlike the classical stationary backstepping kernel equations arising in time-invariant settings, equation \eqref{eq:kernel_1} contains a partial derivative with respect to time. This is inherited by the time-dependence of the coefficient $\lambda$, and as consequence, the kernel must be characterized as the solution of a two-dimensional parabolic evolution equation on the triangular domain $\mathcal P$. The presence of the time derivative requires a well-posedness analysis in function spaces that control both spatial and temporal regularity. In particular, since the prescribed-time design involves a time-varying coefficient $\gamma(x,t) = \lambda(x,t) + c(t)$, whose temporal behavior directly affects the evolution of the kernel, suitable regularity assumptions on $\lambda$ are necessary to guarantee existence of a strong solution and continuous dependence on the data.

This motivates our assumption $\lambda \in G_\alpha(0,T;C^0(0,1))$, with $\alpha \in [1,2]$, which ensures sufficient control of time derivatives of the coefficients. This is essential to obtain a strong solution $k(x,y,t)$ that inherits Gevrey regularity in time and remains well defined up to $t = T$, despite the possible singular behavior of $c(t)$ in prescribed-time stabilization. 

Kernel equations with space- and time-dependent coefficients have been rigorously investigated in \cite{meurer2009tracking,vazquez2008control,jadachowski2012efficient}. In particular, \cite{meurer2009tracking} establishes existence of a strong solution via an integral-operator formulation and a successive approximation argument, and proves that the resulting kernel is of Gevrey order $\alpha \leq 2$ with respect to time. Based on these results, we can deduce the following lemma.

\begin{lemma}\label{Bound} 
Assume that $\lambda \in G_\alpha(0,T; C^0(0,1))$ for some parameter $\alpha\in[1,2]$. Let 
\begin{align*}
    \gamma(x,t) = \lambda(x,t) + c(t)    
\end{align*}
with $c \in C^\infty(0,T)$. Then the kernel equation \eqref{eq:kernel_1} admits a unique strong solution
\begin{align*}
    k \in G_\alpha(0,T; C^2(\Omega)),    
\end{align*}
where 
\begin{align}\label{eq:Omega}
    \Omega = \{(x,y)\in \mathbb{R}^2\,:\, 0 \leq y \leq x \leq 1 \}.
\end{align}
Moreover, there exits a constant $C>0$, such that
\begin{align}\label{eq:k_sup_norm}
    \|k(\cdot,\cdot,t)\|_\infty\leq C
\end{align}
\end{lemma}

\begin{proof}
The proof follows immediately from \cite[Theorems 8 and 9]{meurer2009tracking} and is omitted here for the sake of brevity.
\end{proof}

Notice that, using \eqref{eq:U} and Lemma \eqref{Bound}, one can readily prove that 
\begin{align}\label{eq:U_reg}
    U\in C^\infty(0,T)    
\end{align}

Finally, it is known (see \cite{krstic2008boundary,smyshlyaev2005control}) that the backstepping transformation \eqref{eq:backstepping} has an inverse 
\begin{align}\label{eq:inverse_backstepping}
	v(x,t) = \mathcal{L}[w(\cdot,t)](x) = w(x,t)+\int_0^x l(x,y,t)w(y,t)\,dy 
\end{align}
where the inverse kernel $l(x,y,t)$ satisfies the following PDE system
\begin{align}\label{eq:kernel_2}
	\begin{cases}
		l_t =\theta\big(l_{xx}- l_{yy}\big)  +\gamma l, & (x,y,t)\in\mathcal P
		\\
		l_y(x, 0, t) = ql(x,0,t)  
		\\[5pt]
		\displaystyle l(x, x, t) = -\frac{1}{2\theta}\int_0^x\gamma(\xi,t)\,d\xi
	\end{cases}
\end{align}

Moreover, an analogous result as in Lemma \ref{Bound} can be established also fro the solution $l$ of \eqref{eq:kernel_2} and, in particular, we have the estimate 
\begin{align}\label{eq:l_sup_norm}
    \|l(\cdot,\cdot,t)\|_\infty\leq C.
\end{align}

\section{DeepONet approximation of the time-varying kernel}\label{sec:NO}

To implement the controller \eqref{eq:U}, it is necessary to compute the time-varying kernels $k(x,y,t)$ by solving \eqref{eq:kernel_1}. This, however, is most often computationally expensive, especially for systems with high spatial resolution. This computational burden limits the applicability of backstepping-based controllers in real-time scenarios.

To overcome this challenge, we leverage the DeepONet to learn the mapping from system parameters to the kernel functions. Once trained, DeepONet enables rapid prediction of $k$, significantly reducing the online computation time and making real-time implementation feasible. Moreover, the following result, shows that DeepONet can indeed produce a NO the effectively approximates our kernel functions. 
\begin{theorem}\label{th:NO}
Fix a compact set $A\subset G_\alpha\big(0,T;C^0(0,1)\big)$ and define the operator 
\begin{align*}
    & \mathcal{K}: A \mapsto G_\alpha\big(0,T;C^2(\Omega)\big)
    \\
    &\mathcal{K}(\lambda)(x,y,t) = k(x,y,t),   	
\end{align*}
with $\Omega$ given by \eqref{eq:Omega} and where $k(x,y,t)$ is a solution of \eqref{eq:kernel_1}. Then, for all $\epsilon> 0$, there exists a NO 
\begin{align*}
    \mathcal{\hat{K}}: \Omega\to G_\alpha\big(0,T;C^2(0,1)\big)    
\end{align*}
such that for all $(x,y,t) \in \mathcal{P}$ 
\begin{align}\label{eq:NO_est1}
    \|\mathcal{K}(\lambda)-\hat{\mathcal{K}}(\lambda)\|_{G_\alpha}\leq\epsilon.
\end{align}
\end{theorem}

\begin{proof}
Since Lemma \ref{Bound} ensures the continuity of the operator $\mathcal{K}$, the result follows immediately from \cite[Theorem 2.1]{Approximation}.
\end{proof}

Once we have learned the NO $\hat{\mathcal{K}}$, we can define the following NO-approximated nominal control law
\begin{align}\label{eq:U_NO}
	\hat{U}(t)=\int_{0}^{1}\hat{k}(1,y,t)v(y,t)dy
\end{align}
and use it to stabilize our system. Notice that also for this NO-control, the same as for the original one $U(t)$, we have 
\begin{align}\label{eq:U_hat_reg}
    \hat U\in C^\infty(0,T)    
\end{align}

Moreover, the effectiveness of \eqref{eq:U_NO} as a controller will be demonstrated in the next Section \ref{sec:NO_stabilization}. 

\subsection{Lyapunov analysis under the NO-approximated control law $\lambda\mapsto k$}\label{sec:NO_stabilization}
In this section, we prove that the NO-approximated control law \eqref{eq:U_NO} can stabilize the original system \eqref{eq:heat_1}. More precisely, we show the following sufficient condition for stabilization.

\begin{theorem}\label{theorem1}
There exits constant $\epsilon^*>0$ such that, for all NO approximation $\hat{k}$ of accuracy $\epsilon \in (0,\epsilon^*)$ provided by Theorem \ref{th:NO}, the feedback control \eqref{eq:U_NO} stabilizes the system \eqref{eq:heat_1}, i.e. for any $T>0$ and any initial condition $v_0\in L^2(0,1)$ the associated solution $v$ satisfies \eqref{eq:stab}.
\end{theorem}

\begin{proof}
First of all, notice that when imposing $v(1,t)=\hat U(t)$ in \eqref{eq:backstepping}, with $\hat U(t)$ given by \eqref{eq:U_NO}, we obtain the perturbed target system 
\begin{align}
    &w_t(x,t) = \theta w_{xx}(x,t) - c(t)w(x,t)   \label{eq:target_NO_1}
    \\
    &w_x(0,t) = qw(0,t)  
    \\
    &w(1,t) = -\int_0^1 \Big(\hat{k}(1,y,t)-{k}(1,y,t)\Big)v(y,t)\,dy. \label{eq:target_NO_3}
\end{align}
Let us define the following Lyapunov function
\begin{align}\label{eq:V}
    V(t)=\dfrac{1}{2}\int_0^1 w^2(x,t)\,dx,
\end{align}
whose derivative satisfies
\begin{align*}
    \dot{V} = \int_0^1 w(x,t)w_t(x,t)\,dx = \int_0^1 w(x,t)[\theta w_{xx}(x,t)-c(t)w(x,t)]\,dx. 
\end{align*}
By using integration by parts, \eqref{eq:target_NO_3}, and \eqref{eq:NO_est1}, we get that
\begin{align*} 
    \dot{V} =& \, \theta \int_0^1 w(x,t)dw_x -c(t)\int_0^1w^2(x,t)\,dx
    \\
    =&\,\theta w(1,t)w_x(1,t)-\theta qw^2(0,t)-\theta\int_0^1 w_x^2(x,t)\,dx -c(t)\int_0^1w^2(x,t)\,dx
    \\
    \leq &\, \theta \epsilon \left(\int_0^1 v(y,t)\,dy\right) w_x(1,t)-\theta q w^2(0,t) -\theta \int_0^1w_x^2(x,t)\,dx  - c(t)\int_0^1w^2(x,t)\,dx.
\end{align*}
Using Young's inequality, for any $\delta>0$ we can estimate
\begin{align}\label{eq:V_2}
    \dot{V} \leq \theta \Big( \dfrac{\epsilon^2}{2\delta}\|v(\cdot,t)\|^2 +\dfrac{\delta}{2}w_x^2(1,t)\Big)-\theta\int_0^1w_x^2(x,t)\,dx - c(t)\int_{0}^{1}w^2(x,t)\,dx.
\end{align}


Now, using \eqref{eq:backstepping}, \eqref{eq:k_sup_norm}, \eqref{eq:inverse_backstepping}, \eqref{eq:l_sup_norm}, and the Cauchy-Schwarz inequality, we have that there exists a positive constant $C>0$ such that
\begin{align}\label{eq:bound_vw}
    \|v(\cdot,t)\|_2^2 \leq C\|w(\cdot,t)\|_2^2 
\end{align}
Replacing \eqref{eq:bound_vw} into \eqref{eq:V_2} we get 
\begin{align}\label{eq:V_3}
    \dot{V} \leq \theta \left(\dfrac{\epsilon^2C}{2\delta}\|w(\cdot,t)\|^2 +\dfrac{\delta}{2}w_x^2(1,t)\right)-\theta\int_0^1w_x^2(x,t)\,dx - c(t)\int_{0}^{1}w^2(x,t)\,dx
\end{align}
Moreover, since we are considering strong solutions $w$, we have
\begin{align*}
    w(\cdot,t)\in H^2(0,1)\quad \text{for all } t\in[0,T),
\end{align*}
and therefore $w_x(\cdot,t)\in H^1(0,1)$. The one–dimensional trace theorem yields the existence of a constant $C_{\mathrm{tr}}>0$, independent of $t$, such that
\begin{align}\label{eq:trace_est}
    |w_x(1,t)|^2 \leq C_{\mathrm{tr}} \Big(\|w_x(\cdot,t)\|^2_2 + \|w_{xx}(\cdot,t)\|^2_2\Big).
\end{align}
But since $w$ satisfies \eqref{eq:target_U_NO_1}, we obtain
\begin{align}\label{eq:trace_est2}
    \|w_{xx}(\cdot,t)\|_2 \leq \frac{1}{\theta}\Big(\|w_t(\cdot,t)\|_2 + |c(t)|\|w(\cdot,t)\|_2\Big).
\end{align}
Besides, by parabolic regularity, there exists a constant $C_0>0$ such that
\begin{align*}
    \|w_t(\cdot,t)\|_2 \leq C_0\Big(\|w_x(\cdot,t)\|_2 + \|w(\cdot,t)\|_2\Big).
\end{align*}
Combining this estimate with \eqref{eq:trace_est} and \eqref{eq:trace_est2}, we deduce the existence of a constant $b>0$ such that
\begin{equation}\label{eq:trace}
    |w_x(1,t)|^2 \leq b \Big(\|w_x(\cdot,t)\|_2^2 + \|w(\cdot,t)\|_2^2 \Big).
\end{equation}
Substituting \eqref{eq:trace} into \eqref{eq:V_3}, we then get 
\begin{align*}
    \dot{V} \leq \left(\dfrac{\theta\epsilon^2C}{2\delta}+\dfrac{\theta b  \delta}{2}-c(t)\right)\|w(\cdot,t)\|_2^2 +\left(\dfrac{\theta b \delta}{2}-\theta\right)\|w_x(\cdot,t)\|_2^2
\end{align*}
Choose now $\delta\leq 2/b$, so that
\begin{align*}
    \frac{\theta \delta b}{2}-\theta \leq 0,
\end{align*}
and define 
\begin{align*}
    \epsilon^* = \sup_{t \in [0,T)} \sqrt{\dfrac{2(c(t)-\theta)}{C\theta}}
\end{align*}
Then, for all $\epsilon \in (0,\epsilon^*)$, we obtain that 
\begin{align*}
    \dot{V}(t) \leq &-m_1 V(t)<0
\end{align*}
with 
\begin{align*}
    m_1=2\left(c(t)-\frac{\theta\epsilon^2C}{2}-\theta\right).
\end{align*}
This leads to 
\begin{align*}
    V(t) \leq V(0)\exp\left[-2\left(c(t)-\frac{\theta\epsilon^2C}{2}-\theta\right)t\right]
\end{align*}
and
\begin{align*}
    V(w(\cdot,t)) \leq V(w_0)\exp\left[-2\int_0^tc(\tau)d \tau + (\theta \epsilon^2C+2\theta)t\right].
\end{align*}
Let now  
\begin{align}\label{eq:zeta(t)}
    \displaystyle \zeta(t) = \exp\left(-2\int_{0}^{t}c(\tau)d \tau\right)
\end{align}

Then, from \eqref{eq:c(t)} it follows that $\zeta(t)$ is a monotonically decreasing function such that $\zeta(0)=1$ and $\zeta(t)\to 0$ as $t\to T$. Thus, for all $t \in (0,T)$
\begin{align*}
    V(w(\cdot,t)) \leq V(w_0)\zeta(t)\exp\Big((\theta \epsilon^2C+2\theta)t\Big) .
\end{align*}
Hence, 
\begin{equation}\label{eq:w_1}
    \|w(\cdot,t)\|_2\leq\sqrt{\zeta(t)}\exp\left(\frac{(\theta \epsilon^2C_1+2\theta)}{2}t\right) \|w_0\|_2.
\end{equation}
Finally, combining \eqref{eq:bound_vw} and \eqref{eq:w_1}, we have
\begin{align*}
    \|v(\cdot,t)\|_2 \leq C \sqrt{\zeta(t)} \exp\left(\frac{(\theta \epsilon^2C+2\theta)}{2}t\right)\|v_0\|_2,
\end{align*}
which immediately yields \eqref{eq:stab} thanks to the properties of $\zeta$. 
\end{proof}

\section{Approximating the full feedback law $(\lambda, v)\mapsto U$}\label{sec:NO-approximated backstepping control law} 

In Section \ref{sec:NO}, we have shown how using DeepONet to approximate the kernel $k(x,y,t)$ preserve the backstepping structure and allows to stabilize the PDE system \eqref{eq:heat_1}. However, the control law \eqref{eq:U} still requires integration of this kernel, which could still generate a computational bottleneck. 

On the other hand, learning the full feedback law $(\lambda, v)\mapsto U$ would allow the control to be generated directly from the current state and system parameters, thereby avoiding repeated kernel evaluations and reducing the online computational burden. In addition, this approach may offer increased robustness with respect to state measurement noise and numerical approximation errors, as the operator is trained end-to-end on closed-loop input–output data.  

In this section, we analyze this proposed approach ans show that, although it does not allow to get a full stability of \eqref{eq:heat_1} in the sense of \eqref{eq:stab}, it still yields a \textit{practical stability} stability result and steer the PDE state to an neighborhood of the equilibrium at time $T$. As we shall see, this weaker form of stability arises because we are training together a multiplicative gain $\mathcal{K}(\lambda)$ and a feedback for $v$, generative an extra additive term in the approximation error. 

We begin by proving that the map $(\lambda, v)\mapsto U$ is Lipschitz continuous.

\begin{lemma}
The map 
\begin{align*}
    &\mathcal{G}: G_\alpha(0,T;C^0(0,1))\times L^2((0,1)\times(0,T)) \to C^\infty(0,T)
    \\
    &\mathcal{G}(\lambda,v) = U, 
\end{align*}
with $U$ given by \eqref{eq:U}, is Lipschitz continuous, that is, there exists a positive constant $C>0$ such that, for all 
\begin{align*}
    (\lambda_1,v_1), (\lambda_2,v_2)\in G_\alpha(0,T;C^0(0,1))\times L^2((0,1)\times(0,T)),
\end{align*}
\begin{align*}
    \|U_1-U_2\|_{C^\infty} = \|\mathcal{G}(\lambda_1,v_1)-\mathcal{G}(\lambda_2,v_2)\|_{C^\infty} \leq C\Big(\|\lambda_1-\lambda_2\|_{G_\alpha} + \|v_1-v_2\|_{2,T}\Big).
\end{align*}
\end{lemma}

\begin{proof}
In what follows, we shall denote by $C>0$ a generic positive constant whose value may change from line to line. 

Let $(\lambda_1,v_1)$, $(\lambda_2,v_2)$ be two sets of admissible functions, and define
\begin{align*}
    U_1(t) = \int_0^1 k_1(1,y,t)v_1(y,t)\,dy, 
    \\
    U_2(t) = \int_0^1 k_2(1,y,t)v_2(y,t)\,dy,
\end{align*}
where $k_1$ and $k_2$ are the kernels corresponding to $\lambda_1$ and $\lambda_2$, respectively. Then, for all $t\in(0,T)$ we have
\begin{align*}
    |U_1(t) - U_2(t)| \leq I_1(t) + I_2(t),
\end{align*}
with 
\begin{align*}
    & I_1(t) := \left|\int_0^1 k_1(1,y,t) \big(v_1(y,t)-v_2(y,t)\big)\,dy \right| 
    \\
    & I_2(t):= \left|\int_0^1 \big(k_1(1,y,t)-k_2(1,y,t)\big) v_2(y,t)\,dy \right|. 
\end{align*}
Using the H\"older inequality and the fact that $k$ is bounded, we obtain
\begin{align*}
    I_1(t) \le \|k_1(1,\cdot,t)\|_\infty \, \|v_1(\cdot,t)-v_2(\cdot,t)\|_1 \le C \|v_1(\cdot,t)-v_2(\cdot,t)\|_2.
\end{align*}
For the second term $I_2$, given the continuous dependence of $k$ on $\lambda$, we have 
\begin{align*}
    \|k_1(1,\cdot,t) - k_2(1,\cdot,t)\|_\infty \leq C \|\lambda_1 - \lambda_2\|_{G_\alpha}.    
\end{align*}
Then, 
\begin{align*}
    I_2(t) \leq \|v_2(\cdot,t)\|_2 \, \|k_1(1,\cdot,t) - k_2(1,\cdot,t)\|_2 \leq C \|\lambda_1 - \lambda_2\|_{G_\alpha}.
\end{align*}
Combining the two terms, we therefore obtain that for all $t\in(0,T)$
\begin{align*}
    |U_1(t) - U_2(t)| \leq C\Big(\|v_1(\cdot,t) - v_2(\cdot_t)\|_2 + \|\lambda_1 - \lambda_2\|_{G_\alpha}\Big),
\end{align*}
and the proof follows taking the supremum over $t\in(0,T)$.
\end{proof}

Due to the Lipschitz continuity of $\mathcal{G}$, based on the DeepONet approximation \cite[Theorem 2.1]{Approximation}, we get the following result analogous to Theorem \ref{th:NO}. 
\begin{theorem}\label{th:NO_2}
Fix a compact set 
\begin{align*}
    A\subset G_\alpha(0,T;C^0(0,1))\times L^2((0,1)\times(0,T)).    
\end{align*}
For all $\epsilon> 0$, there exists a NO $\mathcal{\hat{G}}: A\to C^\infty(0,T)$ 
\begin{align}\label{eq:NO_est2}
    \|\mathcal{G}(\lambda,v)-\hat{\mathcal{G}}(\lambda,v)\|_{C^\infty}\leq\epsilon.
\end{align}
\end{theorem}
With $\hat{\mathcal{G}}$ given by Theorem \ref{th:NO_2}, we define the NO control
\begin{equation}\label{eq:U_NO_hat}
	\hat{U} = \hat{\mathcal{G}}(\lambda,v),
\end{equation}  
which yields the perturbed target system  
\begin{align}
	w_t(x,t) &= \theta w_{xx}(x,t) - c(t)w(x,t)   \label{eq:target_U_NO_1}
	\\
	w_x(0,t) &= qw(0,t)  \label{eq:target_U_NO_2}
	\\
	w(1,t) &= \hat{\mathcal{G}}(\lambda,v)
	-{\mathcal{G}(\lambda,v)}\label{eq:target_U_NO_3}
\end{align}

We can then prove the following result of semiglobal practical prescribed-time stability.

\begin{theorem}\label{th:semiglobal}
Consider the closed-loop system \eqref{eq:heat_1} under the NO-approximated backstepping control law $\hat{U} = \hat{\mathcal{G}}(\lambda,v)$ given by \eqref{eq:U_NO_hat}. Let $\varepsilon>0$ be given by Theorem \ref{th:NO_2}. Then, for any initial datum $v_0\in L^2(0,1)$, there exists a constant $C>0$ such that, for all $T>0$
\begin{align}\label{eq:v_estimate_2-1}
    \lim_{t\to T} \|v(\cdot,t)\| \leq C\sqrt{T}\epsilon, \quad\text{ a.e. on } (0,1).
\end{align}
\end{theorem}

\begin{proof}
Define the Lyapunov functional
\begin{equation}\label{eq:V_U_Lya}
    V(t) = \frac{1}{2} \int_0^1 w^2(x,t) \, dx.
\end{equation}
Differentiating along solutions of \eqref{eq:target_U_NO_1}-\eqref{eq:target_U_NO_3} gives 
\begin{align*}
    \dot{V}(t) = \int_0^1 w(x,t) w_t(x,t) \, dx = \int_0^1 w(x,t) \big[ \theta w_{xx}(x,t) - c(t) w(x,t) \big] \, dx.
\end{align*}
Integrating by parts, using the boundary conditions \eqref{eq:target_U_NO_2} and \eqref{eq:target_U_NO_3}, and thanks to \eqref{eq:NO_est2}, we obtain
\begin{align}\label{eq:V_U_5}
    \dot{V} =&\, \theta \bigg[ w(1,t) w_x(1,t) - q w^2(0,t) - \int_0^1 w_x^2(x,t) \, dx \bigg] - c \int_0^1 w^2(x,t) \, dx \notag
    \\
    =&\, \theta\Big(\hat{\mathcal{G}}(\lambda,v)(t)
	-{\mathcal{G}(\lambda,v)}(t)\Big) w_x(1,t) - \theta q w^2(0,t) - \theta \|w_x(\cdot,t)\|_2^2  - c \|w(\cdot,t)\|_2^2 \notag
    \\
    \leq&\, \theta\Big(\varepsilon w_x(1,t) - \|w_x(\cdot,t)\|_2^2\Big)  - c \|w(\cdot,t)\|_2^2.
\end{align}
Using Young's inequality, for any $\delta > 0$ and $t\in(0,T)$ we have
\begin{align}\label{eq:V_U_6}
    \theta\varepsilon w_x(1,t) \leq \frac{\theta\varepsilon^2}{2\delta} + \frac{\theta \delta}{2} |w_x(1,t)|^2.
\end{align}

Moreover, following the same procedure as in the proof of Theorem \ref{theorem1}, we can deduce the existence of a constant $b>0$ such that
\begin{align}\label{eq:V_U_7}
    |w_x(1,t)|^2 \leq b \Big(\|w_x(\cdot,t)\|_2 + \|w(\cdot,t)\|_2\Big)
\end{align}
Substituting \eqref{eq:V_U_7} into \eqref{eq:V_U_6}, together with since $\tilde{\mathcal{G}}<\varepsilon$, yields
\begin{align}\label{eq:V_U_8}
    \theta\varepsilon^2 w_x(1,t) \leq \frac{\theta\epsilon^2}{2\delta} + \frac{\theta \delta b}{2} \Big(\|w_x(\cdot,t)\|_2 + \|w(\cdot,t)\|_2\Big).
\end{align}
Plugging \eqref{eq:V_U_8} into \eqref{eq:V_U_5} we then have
\begin{align}\label{eq:V_U_9}
    \dot{V} \leq &\, \frac{\theta\epsilon^2}{2\delta} + \frac{\theta \delta b}{2} \Big(\|w_x(\cdot,t)\|_2 + \|w(\cdot,t)\|_2\Big) - \theta \|w_x(\cdot,t)\|_2^2 - c \|w(\cdot,t)\|_2^2 \notag
    \\
    \leq &\, \frac{\theta\epsilon^2}{2\delta} +\left(\dfrac{\theta \delta b}{2}-\theta\right)\|w_x(\cdot,t)\|^2_2 +\left(\dfrac{\theta \delta b}{2}-c(t)\right) \|w(\cdot,t)\|^2_2. 
\end{align}
Take $\delta\leq 2/b$, so that
\begin{equation}\label{eq:V_U_10}
    \frac{\theta \delta b}{2}-\theta \leq 0.
\end{equation}
Moreover, set
\begin{align*}
    \kappa(t) := c(t)-\frac{\theta \delta b}{2}.
\end{align*}
Since $\|w(\cdot,t)\|^2_2 = 2V$, we obtain from \eqref{eq:V_U_9} the differential inequality
\begin{align}\label{eq:V_U_12}
    \dot{V} \leq -2\kappa(t) V + \frac{\theta}{2\delta} \epsilon^2.
\end{align}
Multiplying \eqref{eq:V_U_12} by 
\begin{align*}
    \mu(t) = \exp\left(\int_0^t 2 \kappa(t)\,d\tau\right)
\end{align*} 
and integrating from $0$ to $t$ gives
\begin{align}\label{eq:V_U_13}
    V(t) \leq \, \exp\left(-\int_0^t 2(c(\tau))\,d\tau+\frac{\theta b \delta}{2}t\right)V(0) + \dfrac{\theta \epsilon^2}{2\delta} \int_0^t\exp\left(-\int_\tau^t 2 \kappa(s)ds\right)\,d\tau.
\end{align}
Set 
\begin{align*}
    \zeta(t)= \exp\left(-\int_0^t c(\tau)\,d\tau\right). 
\end{align*} 

From \eqref{eq:c(t)}, it follows that $\zeta(t)$ is a monotonically decreasing function such that $\zeta(0)=1$ and $\zeta(T)=0$. Thus, for all $t \in (0,T)$
\begin{align*}
    V(w(\cdot,t)) \leq V(0)\zeta(t) \exp\left(\frac{\theta b \delta}{2}t\right) + \dfrac{\theta \epsilon^2}{2\delta} \int_0^t\exp\left(-\int_\tau^t 2 \kappa(s)ds\right)\,d\tau.
\end{align*}
Then, 
\begin{align*}
    \|w(\cdot,t)\|_2\leq &\, \sqrt{\zeta(t)}\exp\left(\frac{\theta b \delta}{2}t\right) \|w_0\|_2 +\epsilon \sqrt{\dfrac{\theta}{2\delta}}  \left[\int_0^t\exp\left(-\int_\tau^t 2 \kappa(s)ds\right)\,d\tau\right]^{\frac{1}{2}}
    \\
    \leq &\, \sqrt{\zeta(t)}\exp\left(\frac{\theta b \delta}{2}t\right) \|w_0\|_2 +\epsilon \sqrt{\dfrac{\theta t}{2\delta}},
\end{align*}
and, since $\zeta(T)\to 0$ as $t\to T$, we get that 
\begin{align*}
    \|w(\cdot,t)\|_2 \to \epsilon \sqrt{\dfrac{\theta T}{2\delta}}, \quad\text{ as } t\to T. 
\end{align*}
From this inequality and \eqref{eq:bound_vw}, we finally obtain that there exists a constant $C>0$ such that
\begin{align}\label{eq:v_estimate_2-1}
    \lim_{t\to T}\|v(\cdot,t)\| \leq C\sqrt{T}\epsilon, \quad\text{ a.e. on } (0,1).
\end{align}
This completes the proof. 
\end{proof}

\begin{remark}
The bound in \eqref{eq:v_estimate_2-1} depends explicitly on both the time horizon $T$ and the NO approximation error $\varepsilon$, and is proportional to $\sqrt{T}\,\varepsilon$. In particular, for fixed $\varepsilon$, the estimate deteriorates as $T$ increases. This reflects the accumulation of approximation effects over the prescribed time interval. However, $\varepsilon$ is a design parameter corresponding to the approximation accuracy of the NO. Therefore, for any fixed time horizon $T$, the residual bound can be made arbitrarily small by selecting a sufficiently accurate neural approximation. In this sense, the degradation induced by larger $T$ can be compensated by reducing $\varepsilon$, at the cost of increasing the size of the training dataset or the neural network complexity. This trade-off justifies the characterization of the closed-loop behavior as \textbf{practical stability}.
\end{remark}

\section{Simulations}\label{sec:simulations}

This section will present and analyze the performance of the proposed NO-based controllers for the considered PDE models: (i) the approximation of the gain kernel map $\lambda \mapsto k(\lambda)$ using DeepONet, which enables a substantial acceleration in computing the time-varying backstepping kernels; and (ii) the approximation of the full feedback law $(\lambda, v)\mapsto U$, for which we evaluate the resulting practical closed-loop stabilization behavior. 

Through a series of numerical experiments, we illustrate how the NO-based approach simultaneously improves computational efficiency and preserves the desired stabilization properties.

\subsection{Simulations’ configuration, dataset generation and NO training} The plant coefficient $\lambda(x,t)$ is defined as the polynomial
\begin{equation}\label{eq:lambda_1}
	\lambda(x,t)
	= 5 + \cos \bigl(\sigma\, \arccos(x)\bigr)
	+ \dfrac{T}{(T-t)^2}.
\end{equation}
with a shape parameter $\sigma$ and 
\begin{align*}
    c(t) =  \frac{2T}{(T-t)^2}.    
\end{align*}

To generate the dataset over which we trained DeepONet, we randomly sampled 1000 values for $\sigma$ with uniform distribution $U(2, 4)$. For each value of $\sigma$, we solved numerically the time-varying kernel PDE on a uniform spatial grid over the triangular domain $0 \leq y \leq x \leq L$, with parameters $dt = 6.25\times10^{-4}\,\mathrm{s}$, $dx = 0.05\,\mathrm{m}$, time $T = 8\,\mathrm{s}$, and domain length $L = 1\,\mathrm{m}$. For this, we employed an Implicit-Explicit: the dissipative term $\theta\,k_{yy}$ is discretized implicitly along the $y$-direction, which allows for stable integration of the stiff diffusive component without restrictive time-step constraints, while, the second spatial derivative $\theta\,k_{xx}$ and the reaction term $-(\lambda + c) k$ are discretized explicitly, facilitating efficient computation. The chosen time and space steps satisfy the Courant--Friedrichs--Lewy (CFL) condition, guaranteeing numerical stability of the explicit components. 

The result is a dataset of 1000 input-output pairs $(\lambda, k)$. With this dataset, we trained DeepONet using an Nvidia RTX 4060 Ti GPU. After 600 epochs of training, the training 
loss of the NO reached $9.09 \times 10^{-7}$.

\subsection{Simulations results: exponential stabilization with NO-approximated kernel $\lambda \mapsto k$}

With the trained DeepONet model, we have computed new backstepping kernels (not included in the training dataset) and we have compared with their analytical counterpart. Table \ref{tab:nopspeedups} displays the computational times of the two approaches, and highlights the evident computational advantage of using DeepONet: as the spatial resolution increases, the NO maintains a nearly constant runtime, achieving speedups of up to several thousand times with respect to the traditional Implicit-Explicit scheme. 
\begin{table}[h] 
	\centering
	\begin{tabular}{lccc}
		\hline
		\textbf{\begin{tabular}[c]{@{}l@{}}Spatial step \\ Size (dx)\end{tabular}} & \textbf{\begin{tabular}[c]{@{}c@{}}Analytical \\ kernel \\ computational \\time (s)\end{tabular}} & \textbf{\begin{tabular}[c]{@{}c@{}}NO Kernel \\ computational \\time (s)\end{tabular}} & \textbf{Speedup} \\ \hline
		$0.01$  & $3.104$   & $0.023$   & $135$x \\
		$0.005$ & $11.805$  & $0.025$   & $472$x \\
		$0.001$ & $181.5$   & $0.036$   & $5263$x \\ \hline
	\end{tabular} %
	\caption{NO speedups over the analytical kernel calculation for various spatial discretizations.}\label{tab:nopspeedups}
\end{table}

We then tested the employment of the NO-computed backstepping kernels to design a controller \eqref{eq:U_NO} to be applied to the original PDE model \eqref{eq:heat_1}. 

For this, we have set $\sigma=3.3$ and first simulated the heat dynamics \eqref{eq:heat_1} over the time interval $t\in[0,20)$, with initial data
\begin{align*}
	v(x,0) = 10.25  x  (1 - x)
\end{align*}
and in the open-loop regime (that is, with $U(t)=0$). 

As shown in Fig.~\ref{open_loop}, the system is unstable, which illustrates the intrinsic instability of the underlying PDE dynamics.
\begin{figure}[h!]
	\centering
	\includegraphics[width=0.5\linewidth]{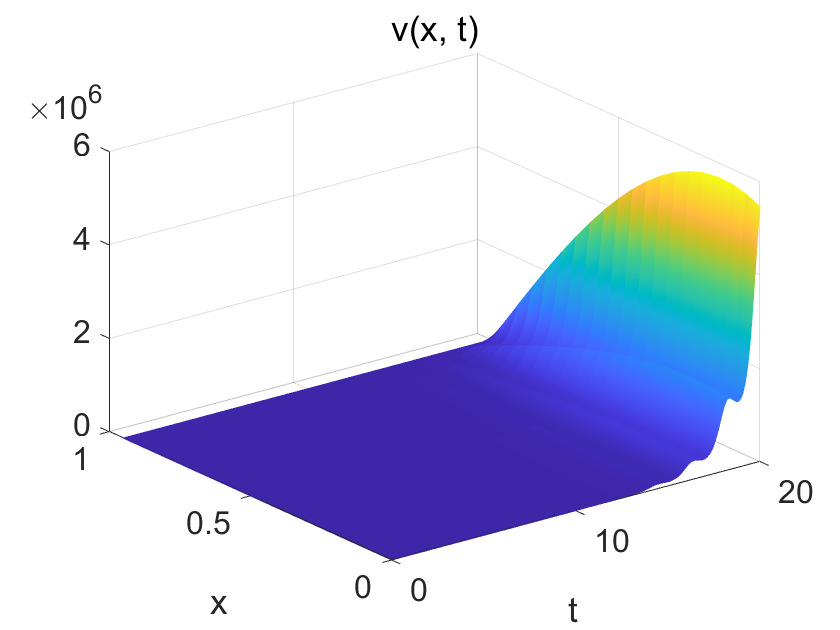}
	\caption{Open-loop evolution of the state $v$}
	\label{open_loop}
\end{figure}

We have then used our NO to compute the backstepping kernel $\hat k$. Fig.~\ref{fig:40-k-imex-m5}–\ref{fig:40-k-imex-m2} show the evolution of these computed kernels at $t=5s$ and $t=7s$, in comparison with the original one solutions of \eqref{eq:kernel_1}. We can observe there that, as time progresses, the kernel profiles remain smooth and exhibit only minor variations, indicating the stability of the kernel-PDE solution and confirming the theoretical result of Lemma \ref{Bound}. Moreover, we can also observe how the $L^2$ approximation error remains uniformly small, below 0.1, as expected from the NO approximation Theorem \ref{th:NO}. 

\begin{figure}[h!]
	\centering
	\includegraphics[width=0.9\linewidth]{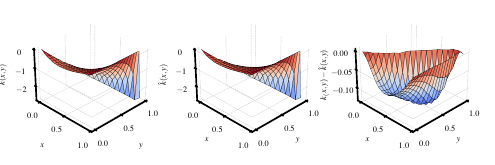}
	\caption{Backstepping kernels $k$ and $\hat k$ at $t = 5$s, together with the $L^2$ approximation error.}
	\label{fig:40-k-imex-m5}
\end{figure}

\begin{figure}[h!]
	\centering
	\includegraphics[width=0.9\linewidth]{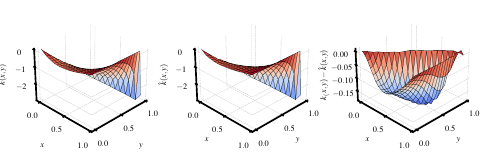}
	\caption{Backstepping kernels $k$ and $\hat k$ at $t = 7$s, together with the $L^2$ approximation error.}
	\label{fig:40-k-imex-m2}
\end{figure}

Finally, with the NO-approximated kernel $\hat k$ we have computed the control $\hat U$ and applied it to our PDE \eqref{eq:heat_1}. We can see in Fig.~\ref{closed_loop} how this time the dynamics is stabilized. 

\begin{figure}[h!]
	\centering
	\includegraphics[width=0.9\linewidth]{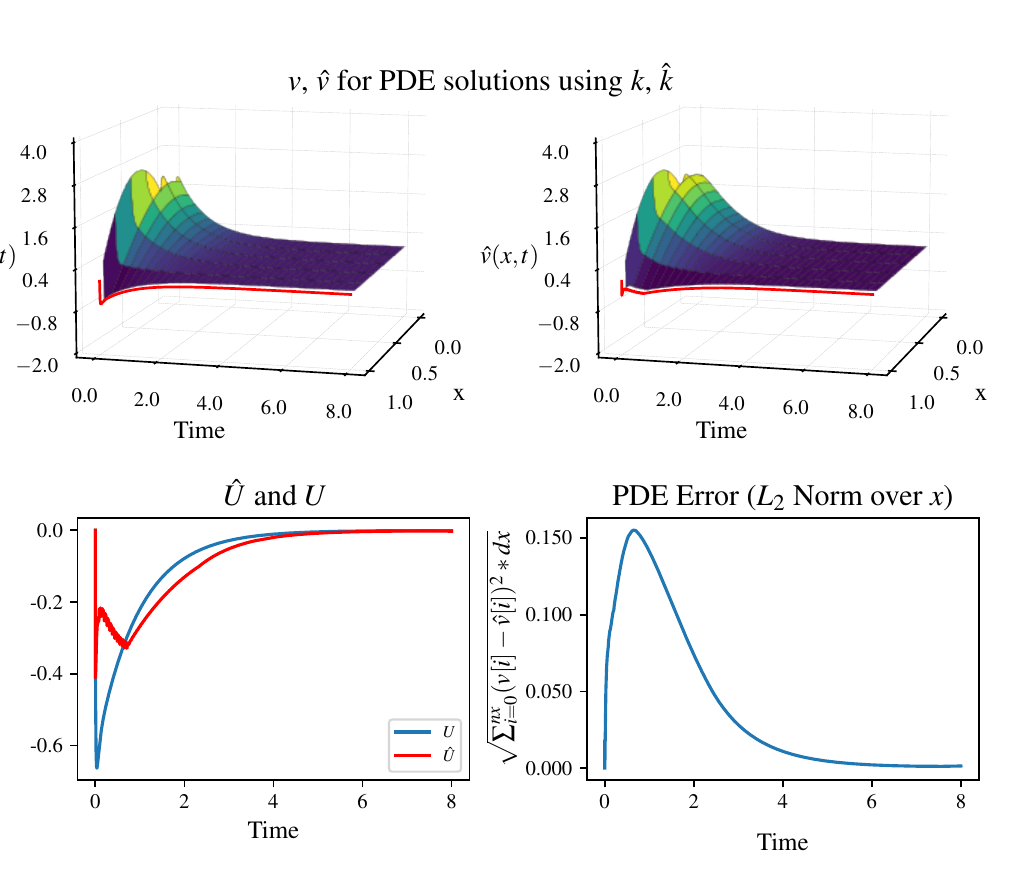}
	\caption{Closed-loop evolution of the true state $v$ and the state $\hat{v}$ obtained using the NO–approximated backstepping kernel $k$. The NO is trained to learn the mapping $\lambda(\cdot,\cdot)\mapsto k$, and the resulting approximate kernel is used to construct the backstepping feedback law.}
	\label{closed_loop}
\end{figure}

This experiment confirms that, in accordance with our main Theorem \ref{theorem1}, DeepONet indeed allows to efficiently compute approximated backstepping kernels that maintain the system's stability. This, together with the speedup evidences in Table \ref{tab:nopspeedups}, demonstrates the efficiency and scalability of the proposed approach for real-time control applications.

\subsection{Simulations results: practical stabilization with NO-approximated full feedback law $(\lambda,v) \mapsto U$}

We conclude by evaluating the practical stabilization performance of the NO–approximated feedback law \eqref{eq:U_NO_hat}. 

The NO is trained offline on a dataset generated from coefficients $\lambda$ given by \eqref{eq:lambda_1} with $\sigma \in {U}(2,4)$ and randomly sampled initial states $v$. Fig.~\ref{fig:semiglobal} shows the closed-loop responses obtained with the learned controller.  

\begin{figure}[h!]
	\centering
	\includegraphics[width=0.9\linewidth]{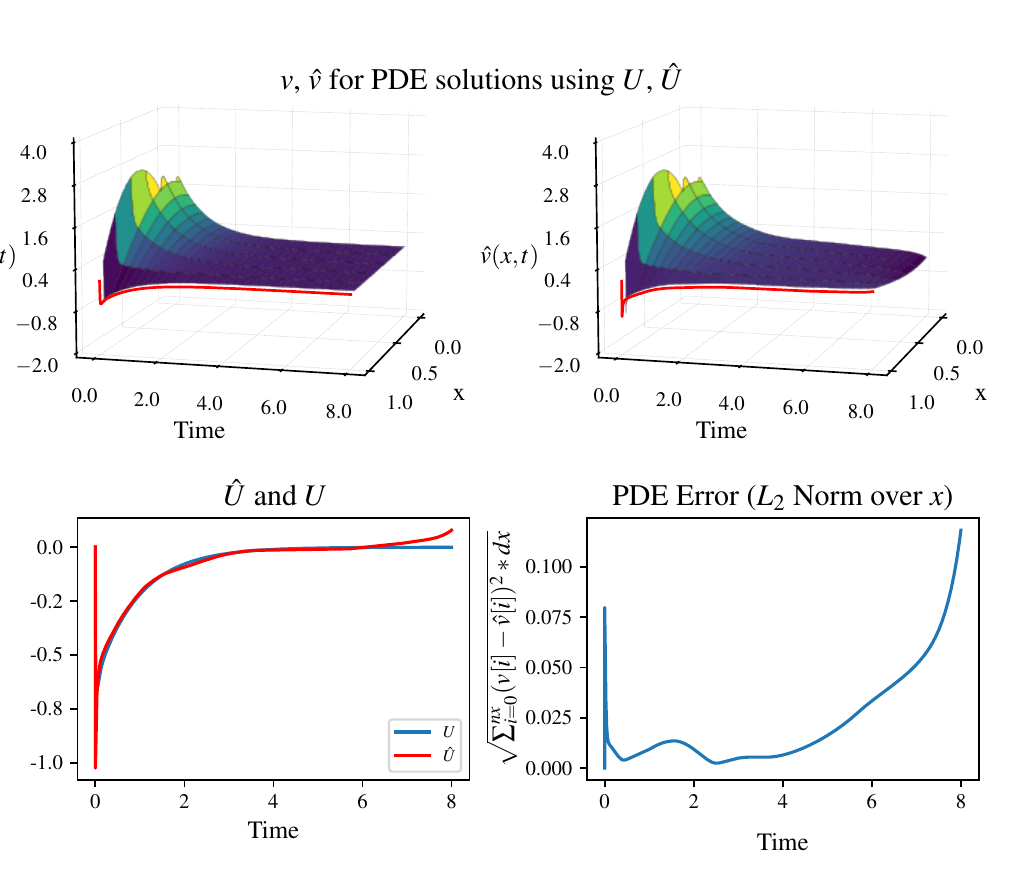}
	\caption{Closed-loop state responses under two control inputs: the ``perfect control'' $U$ computed from the exact backstepping kernel and the ``approximate control'' $\hat{U}$ generated directly by a NO trained to learn the mapping $(\lambda(\cdot,\cdot),\, v) \mapsto U$. }
\label{fig:semiglobal}
\end{figure}

As we can see, while we achieve a practical stabilization of the system, the learned controller generates an approximation error is present in the control signal. As a consequence, the corresponding PDE state exhibits a residual ripple for $T = 8$, whereas the analytically controlled system converges to the equilibrium by $T=8$, in accordance with the target system design.

This behavior is consistent with the theoretical analysis in Theorem~\ref{th:semiglobal}. Moreover, despite not achieving a perfect stabilization at time $T$, we can see the advantage of our proposed method in Table~\ref{tab:fullfeedbackspeedups}, reporting the computational cost of the analytical feedback evaluation and the directly learned feedback law $(\lambda,v)\mapsto U$ for different spatial discretizations. 

\begin{table}[h] 
	\centering
	\begin{tabular}{lccc}
		\hline
		\textbf{\begin{tabular}[c]{@{}l@{}}Step \\ Size \\(dx)\end{tabular}} 
		& 
		\textbf{\begin{tabular}[c]{@{}c@{}}Analytical \\ feedback \\ computational \\ time (s)\end{tabular}} 
		& 
		\textbf{\begin{tabular}[c]{@{}c@{}}Learned feedback \\ computational \\ time (s)\end{tabular}} 
		& 
		\textbf{Speedup } 
		\\ \hline
		$0.01$  & $3.8$   & $0.023$   & $165\times$ \\
		$0.005$ & $12.2$  & $0.025$   & $488\times$ \\
		$0.001$ & $189.8$   & $0.036$   & $5272\times$ \\ \hline
	\end{tabular}
	\caption{Computational speedups obtained by directly learning the full feedback law $(\lambda,v)\mapsto U$ compared to analytical feedback evaluation under different spatial discretizations.}
	\label{tab:fullfeedbackspeedups}
\end{table}

Indeed, the analytical controller requires the numerical evaluation of a spatial integral and computation of kernel at each time step, whose computational complexity increases significantly as the spatial mesh is refined.  In contrast, the learned controller only involves a forward pass through the neural network, whose computational cost remains nearly constant with respect to the spatial resolution. As a result, the speedup becomes more pronounced for finer discretizations, reaching three orders of magnitude when $dx=10^{-3}$. This confirms that directly learning the feedback law provides substantial computational advantages for real-time implementation. 

\section{Conclusions and open problems}\label{sec:conclusions}
This paper addressed the prescribed-time boundary stabilization of a heat equation with spatially and temporally varying reaction coefficient. The prescribed-time backstepping design leads to a genuinely time-evolving kernel governed by a two-dimensional parabolic PDE, whose repeated online solution renders classical implementations computationally prohibitive.

To overcome this bottleneck, we proposed a NO-based framework to approximate the mapping from the time-dependent coefficient to the corresponding backstepping kernel. We established that, under suitable approximation accuracy, the resulting NO-based controller preserves the prescribed-time stability guarantees of the nominal backstepping design. In addition, we investigated direct approximation of the full feedback mapping and proved semiglobal practical prescribed-time stability in this implementation-oriented setting. Numerical experiments confirmed that the proposed approach achieves several orders of magnitude reduction in computational cost, enabling real-time implementation without compromising stability guarantees.

The results of this work suggest several directions for further research. 
\begin{itemize}
    \item[1.]\textbf{Model generalization}: extending the analysis to nonlinear and multi-dimensional parabolic systems would clarify the scalability of neural-operator-based backstepping in more general settings.

    \item[2.]\textbf{Uncertainty and robustness}: robustness with respect to disturbances, measurement noise, and parametric uncertainty remains to be systematically investigated, particularly in the context of operator-approximation errors interacting with closed-loop dynamics.

    \item[3.]\textbf{Large-scale and high-resolution systems}: applying the approach to large-scale PDE systems, including fine spatial discretizations or multi-dimensional domains, where kernel computation dominates the computational cost, and evaluating the efficiency and scalability of the method.

    \item[4.]\textbf{Practical engineering applications}: Implementing the proposed method in real-world systems such as thermal management, chemical reactors, or traffic flow PDEs, to validate its real-time performance and practical feasibility.
\end{itemize}

Overall, this work contributes to bridging operator learning and rigorous infinite-dimensional control design, showing that NO can be incorporated into prescribed-time stabilization frameworks while maintaining explicit stability guarantees.



\end{document}